\documentclass{amsart}[12pt]

\usepackage{amssymb,amsthm}
\usepackage{eucal}
\usepackage[shortlabels]{enumitem}
\usepackage{thmtools,thm-restate}
\usepackage{hyperref}
\usepackage{esvect}
\usepackage{tikz-cd}

\usetikzlibrary{decorations.markings}
\tikzset{negated/.style={
        decoration={markings,
            mark= at position 0.5 with {
                \node[transform shape] (tempnode) {$\backslash$};
            }
        },
        postaction={decorate}
    }
}

\usetikzlibrary{decorations.markings}
\tikzset{boldnegated/.style={
        decoration={markings,
            mark= at position 0.5 with {
                \node[transform shape] (tempnode) {$\textbf{\textbackslash}$};
            }
        },
        postaction={decorate}
    }
}

\setlist[description]{font=\normalfont}

\usepackage{interval}
\intervalconfig{soft  open  fences}
\usepackage{mathtools}

\declaretheorem{theorem}
\declaretheorem[sibling=theorem]{proposition}

\declaretheorem[sibling=theorem]{claim}

\declaretheorem[sibling=theorem]{corollary}
\declaretheorem[style=definition,sibling=theorem]{definition}

\declaretheorem[style=remark,sibling=theorem]{remark}
\declaretheorem[sibling=theorem]{fact}

\DeclareMathOperator{\dom}{dom}

\DeclareMathOperator{\stem}{stem}

\DeclareMathOperator{\ssucc}{succ}
\DeclareMathOperator{\osucc}{osucc}

\renewcommand{\th}{^\text{th}}
\newcommand{\su}{\hspace{1mm}\text{s.t.}\hspace{1mm}}

\newcommand{\seq}[2]{\langle #1 : #2 \rangle}
\newcommand{\pair}[2]{\langle #1 , #2 \rangle}

\newcommand{\Col}{\textup{Col}}

\newcommand{\ON}{\textup{ON}}

\newcommand{\CH}{\textup{\textsf{CH}}}

\newcommand{\LIP}{\textup{\textsf{LIP}}}
\newcommand{\MM}{\textup{\textsf{MM}}}
\newcommand{\BFP}{\textup{\textsf{BFP}}}

\newcommand{\then}{\implies}

\newcommand{\rest}{\upharpoonright}

\newcommand{\Q}{\mathbb Q}

\renewcommand{\P}{\mathbb P}

\newcommand{\pI}{\textup{\textsf{I}}}
\newcommand{\pII}{\textup{\textsf{II}}}

\author{Maxwell Levine}

\title[Classical Namba forcing can have the W.C.A.P.]{Classical Namba forcing can have the weak countable approximation property}

\begin{document}

\maketitle

\begin{abstract} We show that it is consistent from an inaccessible cardinal that classical Namba forcing has the weak $\omega_1$-approximation property. In fact, this is the case if $\aleph_1$-preserving forcings do not add cofinal branches to $\aleph_1$-sized trees. The exact statement we obtain is similar to Hamkins' Key Lemma. It follows as a corollary that $\MM$ implies that there are stationarily many indestructibly weakly $\omega_1$-guessing models that are not internally unbounded. This answers a question of Cox and Krueger and partially answers another. Our result on $\MM$ gives a short proof of a weakening of Cox and Krueger's main result by removing their use of higher Namba forcings, but we find another application of their ideas by answering a question of Adolf, Apter, and Koepke on preservation of successive cardinals by singularizing forcings.\end{abstract}

\section{Background}

Research in infinitary combinatorics has shown that the specific cardinals $\aleph_0$, $\aleph_1$, $\aleph_2$, etc$.$ exhibit distinct properties. One way to look at this is to examine to what extent these cardinals can be turned into each other by forcing. Bukovsk{\'y} and Namba independently showed that $\aleph_2$ can be turned into an ordinal of cofinality $\aleph_0$ without collapsing $\aleph_1$, and this forcing and its variants for other cardinals are now known as Namba forcing \cite{Namba1971}. This paper is about the functions added by variants of Namba forcing.\footnote{The reader is assumed to have familiarity with the basics of forcing theory \cite{Jech2003}.}

The conditions in \emph{classical Namba forcing},\footnote{There is some disorder regarding the referent in the literature because Jech refers to $\P_\mathsf{CNF}$ as Namba forcing \cite[Chapter 15]{Jech2003} and Shelah denotes is $\mathsf{Nm}$ \cite{PIF}, although the version here is due to Bukovsky \cite{Bukovsky1976}. Namba's version \cite{Namba1971}, denoted by Shelah as $\mathsf{Nm'}$, has a different splitting style. For arguments along the lines of \autoref{maintheorem} for the other splitting style, see joint work with Mildenberger \cite[Section 2.1]{Levine-Mildenber2024}.} which we denote $\P_\textup{\textsf{CNF}}$, are perfect trees of height $\omega$ and width $\aleph_2$. Specifically, $p \in \P_\textup{\textsf{CNF}}$ if and only if: $p \subseteq {}^{<\omega} \aleph_2$; $t \in p$ and $s \sqsubseteq t$ implies $s \in p$; for all $t \in p$, $|\{\alpha < \aleph_2:t {}^\frown \langle \alpha \rangle \in p\}| \in \{1,\aleph_2\}$; and for all $t \in p$ there is some $s \sqsupseteq t$ such that $\{\alpha < \aleph_2:s {}^\frown \langle \alpha \rangle \in p\}$ is unbounded in $\aleph_2$. We have $p \le_{\P_\mathsf{CNF}} q$ if and only if $p \subseteq q$. The particularities of the width $\aleph_2$ will be significant for this paper.


Following work of Viale, Weiss, and others, Cox and Krueger introduced the weak guessing model property to explore questions around guessing models. Their approach required an analysis of higher Namba forcings---meaning variations of Namba forcing that add sequences to cardinals above $\aleph_2$---and a demonstration that they have the weak countably approximation property \cite{Cox-Krueger2018}. As they point out, $\CH$ implies that classical Namba forcing cannot have the weak $\omega_1$-approximation property. (A forcing $\P$ has the \emph{weak $\omega_1$-approximation property} if it does not add new functions with domain $\omega_1$ whose initial segments are in the ground model.) Since it collapses $\omega_2^V$ to $\omega_1$, there is a new subset of $\omega_1$, and since $\CH$ implies that Namba forcing does not add reals \cite[Chapter 28]{Jech2003}, it follows that initial segments of this new subset are in the ground model. They asked whether classical Namba forcing could have the weak $\omega_1$-approximation property \cite[Question 5]{Cox-Krueger2018}, and we provide an affirmative answer here.

For the purpose of stating our main theorem, we will say that the \emph{Baumgartner Freezing Property} holds if for all $\omega_1$-preserving forcings $\P$ and all $\omega_1$-sized trees $T$, $\P$ does not add a cofinal branch to $T$. We denote this $\BFP$. (The attribution will be clarified later.) It consistently holds in the presence of certain specializing functions, but we will refer to its abstract form \cite{Baumgartner1983}.

An $\omega_1$-sized tree $T$ is \emph{$B$-special} if there is a function $f:T \to \omega$ such that for all $x,y,z \in T$, if $x \le y,z$ and $f(x)=f(y)=f(z)$, then $y$ and $z$ are compatible. In this case $f$ is called a \emph{$B$-specializing function}.

\begin{proposition}\label{freezing} If all $\omega_1$-sized trees have a $B$-specializing function, then $\BFP$ holds.\end{proposition}

Baumgartner obtained the consistency of $\BFP$ using a model in which there are no Kurepa trees.

\begin{fact}\cite[Section 8]{Baumgartner1983}\label{Baumgartner-theorem} $\BFP$ is consistent from an inaccessible cardinal.\end{fact}





\begin{theorem}\label{maintheorem} Suppose that every $\omega_1$-sized tree cannot gain a cofinal branch from an $\omega_1$-preserving forcing. Then it follows that if $\P=\P_\textup{\textsf{CNF}}$ is the classical Namba forcing and $\dot{\Q}$ is a $\P$-name for a countably closed forcing, then $\P \ast \dot{\Q}$ has the weak $\omega_1$-approximation property.\end{theorem}

Note that $\dot{\Q}$ can name the trivial forcing. The reader will also be able to observe from the proof \autoref{maintheorem} that it is sufficient for $\dot{\Q}$ to be countably strategically closed. \autoref{maintheorem} is optimal in a sense because Krueger showed that any forcing that has the countable approximation property also has the countable covering property \cite{Krueger2019}. Classical Namba forcing cannot have the $\omega_1$-approximation property.

\autoref{maintheorem} also can be seen as a variation on Hamkins' Key Lemma \cite{Hamkins2001}, which was used to show how large cardinal properties are preserved in certain forcing extensions. Many variations have appeared since, notably by Usuba \cite{Usuba2014}. These ideas were used by Viale and Weiss, who introduced guessing models to derive information about the consistency strength of the proper forcing axiom $(\textup{\textsf{PFA}})$ \cite{Viale-Weiss2010}. Guessing models have since become a major topic of research.

Let us introduce some of the basic terminology of guessing models that we will use here. Suppose $M \in P_{\omega_2}(X)$ with $\omega_1 \subseteq M$. We say that $M$ is \emph{weakly $\omega_1$-guessing} if for all $f:\omega_1 \to \ON$ such that $f \rest i \in M$ for all $i<\omega_1$, it follows that $f \in M$. We say that $M$ is \emph{indestructibly weakly $\omega_1$-guessing} if this property holds in all $\omega_1$-preserving extensions. A structure $M$ of cardinality $\aleph_1$ with $\omega_1 \subseteq M$ is internally unbounded if for all $x \in P_{\omega_1}(M)$, there is some $y \in M \cap P_{\omega_1}(M)$ such that $x \subseteq y$. Internal unboundedness and its variations were studied extensively by Krueger and are important for the properties of guessing models.


The findings here may be useful for Viale and Weiss' proposal to work on guessing models for fragments of Martin's Maximum ($\textup{\textsf{MM}}$). Some indication for this possibility comes from an application to another question of Cox and Krueger, who ask whether $\MM$ suffices to prove their main theorem \cite[Question 1]{Cox-Krueger2018}. We are able to obtain $\omega_1$-guessing:

\begin{corollary}\label{maincorollary} $\textup{\textsf{MM}}$ implies that for all $\theta \ge \omega_2$, there is a stationary set of $M \in P_{\omega_2}(H(\theta))$ such that $\omega_1 \subset M$, and $M$ is indestructibly weakly $\omega_1$-guessing and not internally unbounded.\end{corollary}

This in particular uses a weaker large cardinal hypothesis than Cox and Krueger, who obtain their statement from a supercompact and countably many measurables above it. The higher Namba forcings that they use are not needed for obtaining \autoref{maincorollary}. However, we will demonstrate an alternative application for their work:

\begin{theorem}\label{secondtheorem} Assume the consistency of class-many supercompact cardinals. Then there is a forcing extension in which, for all double successor cardinals $\lambda^{++}$, there is a further set forcing extension adding a cofinal $\omega$-sequence to $\lambda^{++}$ without collapsing $\mu$ for $\mu \le \lambda^+$.\end{theorem}


Hence we can have a sense of to what extent these higher versions of Namba forcing generalize the behavior of classical Namba forcing. This answers the first question of a paper by Adolf, Apter, and Koepke \cite{Adolf-Apter-Koepke2018}, who also indicate that a substantial cardinal hypothesis is necessary in order to obtain the statement of \autoref{secondtheorem}.







\section{Classical Namba Forcing and Weak Approximation}



We underscore an important fact that does not depend on $\CH$. The proof of this fact will to some extent be imitated in the proof of \autoref{maintheorem}:

\begin{fact}\label{namba-pres} $\P_\textup{\textsf{CNF}}$ preserves $\omega_1$.\end{fact}

 Now we prove the main theorem.

\begin{proof}[Proof of \autoref{maintheorem}] Suppose for contradiction that $(p',\dot{c}') \in \P \ast \dot{\Q}$ forces that $\dot{F}:\omega_1 \to \ON$ is a $\P \ast \dot{\Q}$-name for a new function whose proper initial segments are in $V$.

\begin{claim}\label{dichotomy-claim} One of the following holds:

\begin{enumerate}

\item $\BFP$ fails.


\item For all $(p,\dot{c}) \le (p',\dot{c}')$ and all $X \in [V]^{\le \omega_1} \cap V$ it is not the case that $(p,\dot{c}) \Vdash \textup{``} \{\dot{F} \rest i: i < \omega_1 \} \subseteq X \textup{''}$.

\end{enumerate}

\end{claim}

\begin{proof} Suppose that \emph{(2)} does not hold. It follows that there is some $(p,\dot{c}) \le (p',\dot{c}')$ and some $X \in [V]^{\le \omega_1} \cap V$ such that $(p,\dot{c}) \Vdash \textup{``}\{\dot{F} \rest i:i<\omega_1\} \subseteq X\textup{''}$. Without loss of generality, there is a large enough $\tau$ such that all elements of $X$ are functions $\sigma:\gamma \to \tau$ for some $\gamma < \omega_1$. We can assume that $|X|>\aleph_0$ since otherwise it would be implied that $(p,\dot{c})$ forces $\dot{F}$ to have domain bounded in $\omega_1$. Therefore, $X$ is a $\omega_1$-sized tree where the ordering is end-extension, i.e$.$ if $y,z \in X$ then $y \le_X z$ if and only if $z \rest \dom y = y$.  Since $\dot{F}$ is forced to be new, $(p,\dot{c}) \Vdash \textup{``}X$ has a cofinal branch not in $V\textup{''}$. Therefore we have shown that $\P \ast \dot{\Q}$, which preserves $\omega_1$ (\autoref{namba-pres} and then countable closure) and adds a cofinal branch to an $\aleph_1$-sized tree, and hence $\BFP$ fails.\end{proof}

We are assuming that $\BFP$ holds, so for the rest of the proof we will argue that Case \emph{(2)} in \autoref{dichotomy-claim} leads to a contradiction.


Now we introduce some standard notation: If $t \in p \in \P$ then $\ssucc_p(t) = \{t' \in p: t' \sqsupseteq t,|t'|=|t|+1\}$ and $\osucc_p(t) = \{\alpha<\aleph_2:t{}^\frown \langle \alpha \rangle \in p\}$. If $p \in \P$, then $\stem p$ is the $\sqsubseteq$-maximal node such that for all $t' \sqsubseteq t$, $|\ssucc_p(t')|=1$.

We will repeatedly use the following in the remainder of the argument:

\begin{claim}\label{name-claim} Suppose $p \Vdash \textup{``}\dot{c} \in \dot{\Q} \textup{''}$ where $\osucc_p(\stem(p))=Z$ and there is a sequence $\seq{(q_\alpha,\dot{d}_\alpha)}{\alpha \in Z}$ such that $(q_\alpha,\dot{d}_\alpha) \le (p \rest \stem(p) {}^\frown \langle \alpha \rangle,\dot{c})$ for all $\alpha \in Z$. If $q = \bigcup_{\alpha \in Z}q_\alpha$, then there is some $\dot{d}$ such that for all $\alpha \in Z$, $q \Vdash \textup{``}\dot{d} \le \dot{c}_\alpha \textup{''}$ and $q_\alpha \Vdash \textup{``}\dot{d} = \dot{c}_\alpha \textup{''}$.\end{claim}

\begin{proof} This is an application of the proof of the Mixing Principle, since $\seq{q_\alpha}{\alpha \in Z}$ is a maximal antichain below $q$.\end{proof}



Now we define the main idea of the rest of the proof. Let $\varphi(i,(q,\dot{d}))$ denote the formula
\begin{align*}
 i < \omega_1 \wedge & (q,\dot{d}) \in \P \ast \dot{\Q}\wedge (q,\dot{d}) \le (p',\dot{c}') \wedge \exists \seq{a_\alpha}{\alpha \in \osucc_q(\stem(q))} \su \\
&\forall \alpha \in \osucc_q(\stem(q)),(q \rest (\stem(q) {}^\frown \langle \alpha \rangle),\dot{d}) \Vdash `\dot{F} \rest i = a_\alpha\textup{'} \wedge \\
& \forall \alpha, \beta \in \osucc_q(\stem(q)), \alpha \ne \beta \then a_\alpha \ne a_\beta.
\end{align*}


\begin{claim}\label{basic-claim} $\forall j<\omega_1,\forall (p,\dot{c}) \le (p',\dot{c}') \in \P \ast \dot{\Q},\exists i \in (j,\omega_1), \exists (q,\dot{d}) \le (p,\dot{c}) \su \stem(p)=\stem(q) \wedge \varphi(i,(q,\dot{d}))$.\end{claim}

\begin{proof}[Proof of Claim] First we establish a slightly weaker claim: $\forall j<\omega_1, \forall (p,\dot{c}) \in \P \ast \dot{\Q}, \forall t \in p \su |\ssucc_p(t)|>1$, there is an $\aleph_2$-sized set $W \subset \osucc_p(t)$ and a sequence $\seq{(q_\alpha,\dot{c}_\alpha),i_\alpha,a_\alpha}{\alpha \in W}\su$:

\begin{itemize}
\item $\forall \alpha \in \osucc_p(t)$,
\begin{itemize}
\item $i_\alpha \in (j,\omega_1)$,
\item $(q_\alpha,\dot{c}_\alpha) \le (p \rest (t {}^\frown \langle \alpha \rangle),\dot{c})$,
\item $(q \rest t {}^\frown \langle \alpha \rangle, \dot{c}_\alpha) \Vdash ``\dot{F} \rest i_\alpha = a_\alpha \textup{''}$,
\end{itemize}
\item $\alpha \ne \beta \then a_\alpha \ne a_\beta$.
\end{itemize}

Consider $(p,\dot{c}) \in \P \ast \dot{\Q}$ and $t := \stem(p)$. Inductively choose a sequence $\seq{(q_{\alpha_\xi},\dot{d}_{\alpha_\xi}),i_{\alpha_\xi},a_{\alpha_\xi}}{\xi<\aleph_2}$ where $\seq{\alpha_\xi}{\xi<\aleph_2} \subseteq \osucc_p(t)$ as follows: Suppose $\seq{(q_{\alpha_\xi},\dot{d}_{\alpha_\xi},a_{\alpha_\xi}),i_{\alpha_\xi}}{\xi<\eta}$ is defined. Then if possible, choose $\beta \in \osucc_p(t) \setminus (\sup_{\xi<\eta}\alpha_\xi)$ and $(q_\beta,\dot{d}_\beta)$ such that $q_\beta \le p \rest t {}^\frown \langle \beta \rangle$ and $(q_\beta,\dot{d}_\beta) \Vdash ``\dot{F} \rest i_\beta = a_\beta\textup{''}$ for some $a_\beta$ with $a_\beta \notin \{a_{\alpha_\xi}:\xi<\eta\}$. Moreover let $i_\beta$ be minimal such that $((q_\beta,\dot{d}_\beta),i_\beta)$ fitting this description can be found. Then set $\alpha_\eta := \beta$. If it is not possible to find such a $\pair{q_\beta}{\dot{d}_\beta}$ and $i_\beta$ and $a_\beta$, then halt the construction.

Suppose for contradiction that the slightly weaker claim fails. This means that the construction in the above paragraph must halt. Let $\eta < \aleph_2$ be the least ordinal where it is not possible to continue the construction in the paragraph above and let $B:=\{a_{\alpha_\xi}:\xi<\eta\}$. Then for all $\alpha \in \osucc_p(t) \setminus (\sup_{\xi<\eta} \alpha_\xi)$ we have $(p \rest t {}^\frown \langle \alpha \rangle,\dot{c}) \Vdash ``\bigcup_{i<\omega_1}\dot{F} \rest i \subseteq B \textup{''}$. Let $q := \bigcup_{\alpha \in \osucc_p(t) \setminus (\sup_{\xi<\eta}\alpha_\xi)}p \rest (t {}^\frown \langle\alpha \rangle)$. Then $q \in \P$ and $q \le p$. Furthermore, it is the case that $(q,\dot{d}) \Vdash ``\bigcup_{i<\omega_1}\dot{F} \rest i \subseteq B \textup{''}$ where $\dot{d}$ is obtained by applying \autoref{name-claim}. But $B$ has size $\le \aleph_1$, contradicting the assumption that we are working in Case \emph{(2)} of \autoref{dichotomy-claim}.


Now we have established the slightly weaker claim. Choose a $\aleph_2$-sized subset $W \subseteq \osucc_p(t)$ and some $i \in (j,\omega_1)$ such that $i_\alpha = i$ for all $\alpha \in W$. Then let $q := \bigcup_{\alpha \in W}q_\alpha$ and apply \autoref{name-claim} to obtain $\dot{d}$.\end{proof}

We plan to build a fusion sequence using \autoref{basic-claim}. To this end, we define a game $\mathcal{G}_k$ for $k<\omega_1$.\footnote{This is where the proof of \autoref{namba-pres} is being imitated (see \cite[Section 2.1, Fact 5]{Cummings-Magidor2011}). Preservation of $\omega_1$ for classical Namba forcing (and for various other Namba forcings) is achieved by showing that the forcing preserves stationary subsets of $\omega_1$, in which case the game involves deciding ordinals in a club subset of $\omega_1$.}


Suppose round $n$ of the game is being played where $n=0$ is the first round. If $n=0$ then let $((q_*,\dot{c}_*),i_*)$ be $((p',\dot{c}'),0)$ (recall that $(p',\dot{c}')$ is the condition from the beginning of the proof of \autoref{maintheorem}). Otherwise if $n>0$ let $((q_*,\dot{c}_*),i_*)$ be $((q_n,\dot{c}_n),i_n)$. First Player $\pI$ chooses an $\aleph_1$-sized subset $Z_n \subseteq \osucc_{q_*}(\stem(q_*))$ and some $\delta_n<k$. Then Player $\pII$ chooses some $\alpha \in \osucc_{q_*}(\stem(q_*)) \setminus Z_n$ and some condition $(q_n,\dot{d}_n) \le (q_* \rest \stem q_* {}^\frown \langle \alpha \rangle,\dot{c}_*)$ and some $i_n \in (\delta_n,k)$ such that $\varphi((q_n,\dot{d}_n),i_n)$ holds. Hence we have the following diagram:


\begin{center}
\def\arraystretch{1.25}
\begin{tabular}{c | c | c | c | c | c | c } 
  \textup{Player \textsf{I}} & $Z_0,\delta_0$ &  & $Z_1,\delta_1$ &  & $Z_2,\delta_2$ &  \\ 
 \hline
  \textup{Player \textsf{II}} &  & $(q_0,\dot{d}_0),i_0$ &  & $(q_1,\dot{d}_1),i_1$ &  & $(q_2,\dot{d}_2),i_2$ 
\end{tabular}
\end{center}

Player $\textsf{\textup{II}}$ loses at some stage $n$ if they cannot find appropriate $(q_n,\dot{c}_n)$ and $i_n$ witnessing $\varphi(i_n,(q_n,\dot{c}_n))$ for some $i<k$, i.e$.$ if they cannot in particular find such $i_n \in (\delta_n,k)$. Otherwise Player $\pII$ wins.
	
\begin{claim}\label{game-claim} There is some $k<\omega_1$ such that Player $\textup{\textsf{II}}$ has a winning strategy in $\mathcal{G}_k$.\footnote{Technically we get club-many such $k<\omega_1$.}\end{claim}

\begin{proof} Suppose this is not the case. For all $i<\omega_1$, $\mathcal{G}_i$ is an open game. Therefore by the Gale-Stewart Theorem, there is a winning strategy $\sigma_i$ for Player $\textup{\textsf{I}}$ in $\mathcal{G}_i$.

Let $\theta$ be large enough for $H(\theta)$ to contain the sets mentioned in the upcoming argument, so we consider the structure $\mathcal{H} := (H(\theta),\in,<_\theta,\P,\dot{F},\seq{\sigma_i}{i<\omega_1})$ where $<_\theta$ is a fixed well-ordering of $H(\theta)$ that allows for Skolem functions. Let $M \prec \mathcal{H}$ be a countable elementary submodel. Then $M \cap \omega_1 \in \omega_1$, so let us denote $k:= M \cap \omega_1$.

We will construct a run of the game $\mathcal{G}_k$ such that Player $\pI$ uses the strategy $\sigma_k$ but nonetheless loses the game. We will define the sequence
\[
 (Z_0,\delta_0),((q_0,\dot{d}_0),i_0),(Z_1,\delta_1), ((q_1,\dot{d}_1),i_1), \ldots
 \]
so that Player $\pII$'s moves are all in $M$ even though $\sigma_k \notin M$.

Assume for notational convenience that the game is defined with the opening move $((p',\dot{c}'),0)$ by Player $\pII$. Suppose we are considering stage $n$ of the game and that $((q_n,\dot{d}_n),i_n)$ is defined (we let it be the opening move if $n=0$). Let $(Z_{n+1},\delta_{n+1})$ be obtained by $\sigma_k$ as applied to the previous moves. Then let $W$ be the set of indices $i$ for which $((p',\dot{c}'),0),\ldots,((q_,\dot{d}_n),i_n)$ is a sequence of Player $\pII$'s moves in the game $\mathcal{G}_i$ where Player $\pI$ is using $\sigma_i$. This set is nonempty and $W \in M$. Therefore if
\[
Y = \bigcup\{Z_i:i \in W, \sigma_i(((q_{-1},\dot{d}_{-1}),i_{-1}),\ldots,((q_,\dot{d}_n),i_n))=(Z_i,\delta^*)\}
\]
then $Y \in M$, and moreover $|Y|<\aleph_2$ because it is the union of at most $\aleph_1$-many $\aleph_1$-sized sets. Choose $\alpha \in \osucc_{q_n}(\stem(q_n)) \setminus Y$ and apply \autoref{basic-claim} to $(q_n \rest \stem(q_n) {}^\frown \langle \alpha \rangle,\delta_{n+1})$ in order to obtain $((q_{n+1},\dot{d}_{n+1}),i_{n+1})$. By elementarity, we can obtain $((q_{n+1},\dot{d}_{n+1}),i_{n+1}) \in M$, so $i_{n+1}<k$, hence we can continue the construction.

Since Player $\pII$ does not lose at any initial stage, they win the run of the game $\mathcal{G}_k$. Hence we have obtained our contradiction.
 \end{proof}


Now we will build a condition $(q,\dot{d}) \in \P \ast \dot{\Q}$ by a fusion process in such a way that any stronger condition deciding $\dot{F} \rest k$ will also code a cofinal sequence in $\omega_2$ that exists in the ground model $V$, thus obtaining a contradiction.

For this we define a bit more standard notation. For all $n<\omega$, define the \emph{$n\th$-order splitting front} of some $p \in \P$ as the set of $t \in p$ such that $|\osucc_p(t)|>1$ and such that there are at most $n$-many $t' \sqsubseteq t$ with $t' \ne t$ that have this property.

Fix a sequence $\seq{\delta_n}{n<\omega}$ converging to $k$. We will define $\seq{(p_n,\dot{c}_n)}{n < \omega}$ by induction on $n<\omega$ in such a way that:

\begin{enumerate}

\item $\seq{p_n}{n<\omega}$ is a fusion sequence,

\item for all $n<\omega$, $p_{n+1} \Vdash \textup{``}\dot{c}_{n+1} \le \dot{c_n}\textup{''}$,

\item For all $n<\omega$, if $A_n$ is the $n\th$-order splitting front of $p_n$, then for all $t \in A_n$, the following is the case: Let $s_0 \sqsubseteq s_1 \sqsubseteq \ldots \sqsubseteq s_n=t$ be the sequence of all splitting nodes up to and including $t$. Then there is a sequence $Z^t_0,\ldots,Z^t_n$ such that
\[
(Z^t_0,\delta_0),((p_0 \rest s_0,\dot{d}_0),i_0),\ldots,(Z^t_n,\delta_n),((p_n \rest s_n,\dot{d}_n),i_n)
\]
is a run of the game $\mathcal{G}_k$ in which Player $\pII$'s moves are determined by the winning strategy obtained in \autoref{game-claim}.
\end{enumerate}

Note that the third point implies the following: For all positive $n<\omega$, if $A_n$ is the $n\th$-order splitting front of $p_n$, then for all $t \in A_n$, there is $i_t \in (\delta_n,k)$ and a sequence $\seq{a_s}{s \in \ssucc_{p_n}(t)}$ witnessing that $\varphi(i_t,(p_n \rest t,\dot{c}_n))$ holds.

We do this as follows: Start with stage $-1$ for convenience and let $(p_{-1},\dot{c}_{-1})=(p,\dot{c})$. Then $A_{-1} = \stem(p_{-1})$. Now assume we have defined $p_{n-1}$, that $A_n$ is its $n\th$-order splitting front, and we are considering $t \in A_n$. Let $s_0 \sqsubseteq s_1 \sqsubseteq \ldots \sqsubseteq s_{n-1} = t$ be the sequence of splitting nodes up to and including $t$. Let $S_t$ be the set of $\alpha \in \osucc_{p_{n-1}}(t)$ such that for some $Z^\alpha_n$, the winning strategy for Player $\pII$ applied to the sequence
\[
(Z^t_0,\delta_0),((p_0 \rest s_0,\dot{d}_0),i_0),\ldots,(Z^t_{n-1},\delta_{n-1}),((p_{n-1} \rest s_{n-1},\dot{d}_{n-1}),i_{n-1}),(Z^\alpha_n,\delta_n)
\]
produces some $((q_n,\dot{d}_n),i_n)$ where $q_n \le p_{n-1} \rest t {}^\frown \langle \alpha \rangle$. We claim that $|S_t|=\aleph_2$. Otherwise Player $\pI$ would have a winning move for the sequence
\[
(Z^t_0,\delta_0),((p_0 \rest s_0,\dot{d}_0),i_0),\ldots,(Z^t_{n-1},\delta_{n-1}),((p_{n-1}\rest
s_{n-1},\dot{d}_{n-1}),i_{n-1})
\]
by playing $S_t$ as the subset-of-$\aleph_2$-component of their move. For each such $t \in A_n$ and $\alpha \in S_t$, choose $(q_{t,\alpha},\dot{d}_{t,\alpha})$ to be produced by the winning strategy for Player $\pII$ as the $\P \ast \dot{\Q}$ component of their move. Now let $p_n = \bigcup \{q_{t,\alpha}:t \in A_n,\alpha \in S_t\}$. Then use \autoref{name-claim} to obtain $\dot{c}_{n}$ from the $\dot{c}_{t,\alpha}$'s.


Now let $q$ be the fusion limit of $\seq{p_n}{n<\omega}$ and let $\dot{d}$ be the name canonically forced by $q$ to be a lower bound for $\seq{\dot{c}_n}{n<\omega}$. Then $(q,\dot{d})$ forces that the generic sequence for $\P$ can be recovered from $\dot{F} \rest k$ as follows: Let $(r,\dot{e}) \le (q,\dot{d})$ force $\dot{F} \rest  k = g \in V$. We can inductively choose a cofinal branch $b \subset r$ such for all $t \in b$, for some $i<k$, $g \rest i_t = a_t$. Specifically, we construct $b$ by defining a sequence $\seq{s_n}{n<\omega}$ of splitting nodes as follows: Let $s_0 = \stem q$. Given $s_n$, let $s^*_{n+1} \sqsupseteq s_n$ be the next splitting node. Then since $\varphi(i_t,(r \rest s^*_{n+1},\dot{e}))$ holds for some $i_t \in (\delta_n,k)$, there is some $\alpha \in \osucc_r(s^*_{n+1})$ such that $(r \rest s^*_{n+1} {}^\frown \langle \alpha \rangle,\dot{e}) \Vdash \textup{``}g \rest i_t = a_t \textup{''}$. Then let $s_{n+1} = s_{n+1}^* {}^\frown \langle \alpha \rangle$. Then let $b = \{t \in r: \exists n<\omega, t \sqsubseteq s_n\}$. This implies that $(r,\dot{e})$ forces that the generic object is equal to $b$, i.e$.$ that $\bigcap \Gamma(\P)=b \in V$, but this is not possible.

Hence $(q,\dot{d}) \Vdash ``\dot{F} \rest k \notin V\textup{''}$ lest we obtain the contradiction from the previous paragraph. This contradicts the premise from the beginning of the proof that initial segments of $\dot{F}$ are in $V$.\end{proof}



Now we can work on proving \autoref{maincorollary}. This is mostly a matter of noting some statements in the established theory. First we state the fact that is most frequently used to produce various sorts of guessing models:

\begin{fact}[Woodin]\label{woodin} (Implicit in \cite[Theorem 2.53]{Woodin1999}) Suppose that $\P$ is a forcing poset and that for all sequences $\mathcal{D}=\seq{D_\alpha}{\alpha<\omega_1}$ of dense subsets of $\P$, there is a $\mathcal{D}$-generic filter. Then for any regular cardinal $\theta$ such that $\P \in H(\theta)$, there are stationarily many $M \in P_{\omega_2}(H(\theta))$ with $\omega_1 \subseteq M$ for which there exists an $(M,\P)$-generic filter.\end{fact}

Then we use a property to violate internal unboundedness. The following is a partial weakening of Cox-Krueger \cite[Lemma 5.2]{Cox-Krueger2018}.

\begin{fact}\label{cox-krueger1} Let $\P$ be a forcing poset such that $\P \in H(\theta)$ for a regular cardinal $\theta \ge \omega_2$. Let $M \in P_{\omega_2}(H(\theta))$ be such that $\omega_1 \subseteq M$ and $M \prec (H(\theta),\in,\P,\tau)$. Suppose that $\P$ forces that there is a countable set of ordinals in $\tau$ not covered by any countable set in $V$. If $G$ is an $M$-generic filter on $\P$, then there is a countable subset of $M$ that is not covered by any countable set \emph{in} $M$.\end{fact}

Then we have another weaking of Cox-Krueger \cite[Proposition 5.4]{Cox-Krueger2018} which an analog of the lemma used by Viale and Weiss to produce guessing models from Woodin's result \cite[Lemma 4.6]{Viale-Weiss2010}.

\begin{fact}\label{cox-krueger2} Fix a regular cardinal $\theta \ge \omega_2$. Assume that the poset $\P$ has the weak $\omega_1$-approximation property and forces that $2^{\theta}$ has size $\omega_1$. Then there exists a set $w$ and a $\P$-name $\dot{\Q}$ for an $\omega_1$-cc forcing poset such that the following holds: For any regular $\chi$ with $\P,\dot{\Q},w \in H(\chi)$ and any $M \in P_{\omega_2}(H(\chi))$ such that $\omega_1 \subseteq M$ and $M \prec (H(\chi),\in,\P \ast \dot{\Q},\theta,w)$, if there exists an $M$-generic filter on $\P \ast \dot{\Q}$, then $M \cap H(\theta)$ is indestructibly $\omega_1$-weakly guessing.\end{fact} 

\begin{proof}[Proof of \autoref{maincorollary}] Fix the $\theta$ such that we want stationarily many indestructibly weakly $\omega_1$-guessing models in $P_{\omega_2}(H(\theta))$ that are not internally unbounded. Let $\P$ denote the classical Namba forcing and let $\dot{\Q}$ be a $\P$-name for the L{\' e}vy collapse $\Col(\omega_1,(2^\theta)^+)$. Let $\chi > \theta$ be large enough for $H(\chi)$ to contain the needed objects referred to in the statement of \autoref{cox-krueger2} and then apply $\MM$ to \autoref{woodin} to find a stationary set $S' \subseteq H_{\omega_2}(H(\chi))$ of $M \supseteq \omega_1$ for which there is an $(M,\P \ast \dot{\Q})$-generic filter. Then the set $S:= \{M \cap H(\theta):M \in S'\}$ is the stationary set we are looking for. Since $\P$ adds a countable sequence in $\omega_2$ not covered by any set in $V$ and since we can restrict to find an $(M,\P)$-generic filter, it follows by \autoref{cox-krueger1} that for all $M \in S'$, $M$ is not internally unbounded as witnessed by a subsequence of $\omega_2$, so this is also the case for $M \cap H(\theta)$. Since $\MM$ implies $\BFP$, \autoref{maintheorem} implies that $\P \ast \dot{\Q}$ has the weak $\omega_1$-approximation property, so it follows by \autoref{cox-krueger2} that all $M \in S$ are indestructibly weakly guessing.\end{proof}

\section{Higher Namba Forcing and Cardinal Preservation}

Now we work towards proving \autoref{secondtheorem}. The work uses higher Laver-Namba forcings employed by Cox and Krueger to obtain a stationary set of guessing models. We retain the crux of their argument: A deft use of the pigeonhole principle, where $\dot{f}$ is a name for a function, $\seq{D_\xi}{\xi \in \dom \dot{f}}$ is a sequence of open dense sets deciding values $\dot{f}$ and a stem length $n<\omega$ is isolated so that unboundedly many values of the function can be decided with a single condition. In Cox and Krueger's case they argue that a certain Laver-Namba forcing has a weak approximation property, but in our case we will be showing that certain cardinals are preserved.

If $\lambda$ is a cardinal and $X \subseteq \mathcal{P}(\lambda)$, let $\P_\textup{\textsf{LNF}}(X)$ be the Laver version of Namba forcing with splitting into members of $X$. This means that $p \in \P_\textup{\textsf{LNF}}$ if and only if $p \subseteq {}^\omega \lambda$ and there is some $t \in p$ such that for all $s \sqsupseteq t$, $\osucc_p(s) \in X$. The main feature of the Laver version to consider is that we can use direct extensions. Given $p, q \in \P_\textup{\textsf{LNF}}(X)$, we write $p \le^* q$ if $p \le q$ and $\stem(p) = \stem(q)$. We will also use a principle found by Laver, for which we will suggest some useful notation:

\begin{definition}\cite[Chapter X,Definition 4.10]{PIF} Given a successor cardinal $\kappa^+$, we write $\LIP(\kappa^+)$ if there is a $\kappa^+$-complete ideal $I \subset P(\kappa^+)$ such that there is a set $\mathcal{D} \subseteq I^+$ that is $\kappa$-closed subset and dense in itself, i.e$.$ for all $A \in I^+$, there is some $B \subseteq A$ with $B \in I^+$ such that $B \in \mathcal{D}$.\end{definition}

\begin{fact}[Laver]\label{gettingLIP} If $\kappa < \mu$ where $\kappa$ is regular and $\mu$ is measurable, then $\Col(\kappa,<\mu)$ forces $\LIP(\mu)$.\end{fact}

Laver's proof of \autoref{gettingLIP} is unpublished, but the argument is similar to the one found by Galvin, Jech, and Magidor for obtaining a certain precipitous ideal on $\aleph_2$ \cite{Galvin-Jech-Magidor1978}. Some additional details appear in Shelah \cite[Chapter X]{PIF}. We will use a version of this result for successive cardinals that also comes from unpublished work of Laver:

\begin{theorem}[Laver]\label{gettingmoreLIP} Assume the consistency of class-many supercompact cardinals. Then there is a forcing extension in which $\LIP(\lambda^{++})$ holds for all infinite cardinals $\lambda$. (See \cite[Page 89]{Hodges-Shelah1981} and see \cite{Shelah1996} for related arguments.)\end{theorem}

%

The following facts are covered in detail by Cox and Krueger, though we state them in some specificity. They are fairly standard arguments for Namba forcings, and the first two facts to some extent go back to Laver's proof of the consistency of the Borel Conjecture (see also Cummings-Magidor \cite{Cummings-Magidor2011} and Shelah \cite{PIF}).

\begin{fact}\label{direct-decision} \cite[Lemma 6.5]{Cox-Krueger2018},\cite[Section 2.1, Fact 1]{Cummings-Magidor2011} Suppose $I \subseteq \kappa^+$ is a $\kappa^+$-complete ideal and suppose $p \in \P_\textup{\textsf{LNF}}(I^+)$ forces that $\dot{\gamma}$ is a name for an ordinal below $\kappa$. Then there is some $q \le^* p$ and $\delta \in \ON$ such that $q \Vdash \textup{``}\dot{\gamma} = \delta \textup{''}$.\end{fact}

\begin{fact}\label{prikry-density} \cite[Lemma 6.4]{Cox-Krueger2018},\cite[Section 2.1, Fact 2]{Cummings-Magidor2011} Suppose $I \subseteq \kappa^+$ is a $\kappa^+$-complete ideal. Let $D \subseteq \P_\textup{\textsf{LNF}}(I^+)$ be dense open. Then for each $p \in \P_\textup{\textsf{LNF}}(I^+)$, there is some $q \le^* p$ and some $n<\omega$ such that for any $t \in q$ with $|t|=\stem(p)+n$, $q \rest t \in D$.\end{fact}

Finally, we have closure of the direct extension, the argument for which is much easier than the one for the last two facts. Lower bounds can be obtained by inductively defining splitting sets via the lower bounds for $\LIP(\kappa^+)$.

\begin{fact}\label{prikry-closure}\cite[Lemma 6.13]{Cox-Krueger2018} Suppose that $\LIP(\kappa^+)$ holds and is witnessed by $I$ and that $\eta<\kappa$. If $\seq{p_\xi}{\xi<\eta}$ is a $\le^*$-decreasing sequence of conditions in $\P_\textup{\textsf{LNF}}(I^+)$, then there is $\bar p$ such that $\bar p \le^* p_\xi$ for all $\xi<\eta$.\end{fact}

\begin{proof}[Proof of \autoref{secondtheorem}] Using \autoref{gettingmoreLIP}, assume that $\LIP(\lambda^{++})$ holds for all infinite $\lambda$. For each infinite cardinal $\mu$ such that $\LIP(\mu)$ holds, let $I_\mu$ be the witnessing ideal and let $\mathcal{D}_\mu$ be the corresponding dense set given by $\LIP(\mu)$. We will argue that for all $\mu = \lambda^{++}$, $\P_\mu:= \P_\textup{\textsf{LNF}}(\mathcal{D}_\mu)$ preserves the cofinalities of all regular $\nu \le \lambda^+$. To do this, we fix some $\lambda^{++}$ and choose some regular $\nu \le \lambda^+$ for which we will prove the statement.

First we consider the possibility that $\nu>\aleph_0$ is forced by some $p \in \P_\mu$ to have countable cofinality. Then if $\dot{f}:\omega \to \nu$ is a name for an unbounded function we can use \autoref{direct-decision} to choose a $\le^*$-descending sequence $\seq{p_n}{n<\omega}$ below $p$ such that $p_n$ forces ``$\dot{f}(n)=\beta_n$'' for some $\beta_n$. Then use \autoref{prikry-closure} to obtain a $\le^*$-lower bound $q$ for this sequence. Then $q \Vdash \text{``} \sup f[\omega]= \sup_{n<\omega}\beta_n < \nu \text{''}$, which is a contradiction.



Now suppose that $\nu$ is forced to have cofinality $\tau$ where $\aleph_0 < \tau < \nu$. Suppose that $p \in \P_\mu$ forces that $\dot{f}$ is a $\P_\mu$-name for a function $\tau^V \to \nu$. If $\dot{f}$ were a surjection, then there would be a strictly increasing and unbounded function, so assume without loss of generality that $\dot{f}$ is forced by $p$ to be strictly increasing and unbounded.

For all $\xi<\tau$, let $D_\xi \subseteq \P_\mu$ be the open dense set of conditions deciding $\dot{f}(\xi)$. Now we define a $\le^*$-decreasing sequence $\seq{p_\xi}{\xi<\tau}$ as follows: Let $p_0 = p$. If $p_\xi$ is defined, use \autoref{prikry-density} to obtain some $p_{\xi+1} \le^* p_\xi$ such that for some $n_\xi<\omega$, for any $t \in q$ with $|t|=\stem(p_\xi)+n_\xi$, we have $q \rest t \in D_\xi$. If $\xi$ is a limit and we have defined $\seq{p_{\xi'}}{\xi'<\xi}$, then let $p_\xi$ be a $\le^*$-lower bound obtained using \autoref{prikry-closure}. Once we have defined the sequence, apply \autoref{prikry-closure} again to obtain $\bar{p}$, a $\le^*$-lower bound of $\seq{p_\xi}{\xi<\tau}$ (if $\nu^+ = \mu$ then this is the best possible $\le^*$-closure).

Now apply the Pigeonhole Principle to find an unbounded set $X \subseteq \tau$ and some $\ell<\omega$ such that for all $\xi \in X$, $n_\xi = \ell$. Take any $t \in \bar{p}$ such that $|t| = |\stem(\bar{p})|+ \ell$ and let $q = \bar{p} \rest t$. Then for all $\xi \in X$, $q \in D_\xi$ and therefore there is some $\beta_\xi$ such that $q \Vdash \textup{``}\dot{f}(\xi) = \beta_\xi \textup{''}$. Let $\beta:=\sup_{\xi \in X}\beta_\xi + 1 < \nu$. Then since $q$ forces that $\dot{f}$ is strictly increasing, it follows that $q$ forces that the range of $\dot{f}$ is bounded by $\beta$. This shows that $\nu$ is preserved by $\P_\mu$ and thus completes the proof.\end{proof}

\begin{remark}\label{precipitousness} $\LIP(\aleph_2)$ implies $\CH$: Suppose $\P=\P_\textup{\textsf{LNF}}(\mathcal{D})$ where $\mathcal{D}$ is the $\LIP$-dense set and $\dot{r}$ is a $\P$-name for a subset of $\omega$. Use \autoref{direct-decision} to build a $\le^*$-descending sequence $\seq{p_n}{n<\omega}$ of conditions deciding $\dot{r}(n)$. Then their lower bound decides $\dot{r}$. Hence $\P$ cannot add reals if $\LIP(\aleph_2)$ holds, but it does add reals under $\neg \CH$.

Hence the nice behavior obtained for $\P_\textsf{\textup{LNF}}(\mathcal{D}_\mu)$ for higher cardinals does not seem adaptable to the situation in \autoref{maintheorem}.\end{remark}

We close with some questions:

\begin{enumerate}

\item What is the exact consistency strength of the statement that $\P_\textsf{\textup{CNF}}$ has the weak $\omega_1$-approximation property? \autoref{maintheorem} shows that an inaccessible cardinal is sufficient.

\item Suppose that $I$ is just the bounded ideal on $\aleph_2$. Is it consistent that the Laver-Namba forcing $\P_\textsf{\textup{LNF}}(I^+)$ has the weak $\omega_1$-approximation property? What about for other ideals $I$ on $\aleph_2$? (See \autoref{precipitousness}.)


\end{enumerate}

\subsection*{Acknowledgements} Thank you to Jouko V{\" a}{\" a}n{\" a}nen for explaining why arguments implicit in his work show that $\LIP(\aleph_2)$ implies $\CH$ \cite{Vaananen1991}. And thank you to Saharon Shelah for pointing me to the work connected to \autoref{gettingmoreLIP}.





\bibliographystyle{plain}
\bibliography{bibliography}

\end{document}